\documentclass{article}
\usepackage[utf8]{inputenc}
\usepackage{amsthm}
\usepackage{amsfonts}
\usepackage{amsmath}
\usepackage{verbatim}
\usepackage{natbib}
\newtheorem{lemma}{Lemma}
\newtheorem{theorem}{Theorem}

\title{Universal Approximation on the Hypersphere}
\author{Tin Lok James Ng \and Kwok-Kun Kwong }
\date{}
\begin{document}

\maketitle

\begin{abstract}
    
It is well known that any continuous probability density function on $\mathbb{R}^m$ can be approximated arbitrarily well by a finite mixture of normal distributions, provided that the number of mixture components is sufficiently large. The von-Mises-Fisher distribution, defined on the unit hypersphere $S^m$ in $\mathbb{R}^{m+1}$, has properties that are analogous to those of the multivariate normal on $\mathbb{R}^{m+1}$. We prove that any continuous probability density function on $S^m$ can be approximated to arbitrary degrees of accuracy by a finite mixture of von-Mises-Fisher distributions.  
\end{abstract}

\section{Introduction}
Finite mixtures of distributions \citep{mclachlan2000} are being widely used in various fields for modelling random phenomena. In a finite mixture model, the distribution of random observations is modelled as mixture of a finite number of component distributions with varying proportions. The finite mixture of normal distributions \citep{fraley2002} is one of the most frequently used finite mixture models for continuous data taking values in the Euclidean space, because of their flexibility of representation of arbitrary distributions. Indeed, it has been shown that given sufficient number of mixture components, a finite mixture of normals can approximate any continuous probability density functions up to any desired level of accuracy \citep{bacharoglou2010, nguyen2019}.
\\\\
Despite the success and popularity of finite mixture of normal distributions in a wide range of applications, frequently, data possess more structure and representing them using Euclidean vectors may be inappropriate. An important case is when data are normalized to have unit norm, which can be naturally represented as points on the unit hypersphere $S^m := \{x \in \mathbb{R}^{m+1}: ||x||_2 = 1 \}$. For example, the direction of flight of a bird or the orientation of an animal can be represented as points on the circle $S^1$ or sphere $S^2$. Consequently, standard methods for analyzing univariate or multivariate data cannot be used, and distributions that take into account the directional nature of the data are required.
\\\\
The von-Mises-Fisher distribution \citep{fisher1993} is one of the most commonly used distribution to describe directional data on $S^m$ which has properties analogous to those of the multivariate normal on $\mathbb{R}^{m+1}$. A unit norm vector $x \in S^m$ has a von-Mises-Fisher distribution if it has density
$$ f_{m+1}(x; \mu, \kappa) = c_{m+1}(\kappa) \exp(\kappa \langle x, \mathbf{\mu} \rangle ) , \quad x \in S^m,$$
where $\kappa > 0$ is the concentration parameter and the mean direction $\mu \in S^m$ satisfies $||\mu||=1$. In particular, as $\kappa$ increases, the distribution becomes increasingly concentrated at $\mu$. The normalizing constant $c_{m+1}(\kappa)$ is given by
$$ c_{m+1}(\kappa) = \frac{\kappa^{ \frac{m+1}{2} - 1 }}{ (2 \pi)^{ \frac{m+1}{2} } I_{\frac{m+1}{2}-1}(\kappa)} ,$$
where $I_v$ is the modified Bessel function at order $v$.
\\\\
A finite mixture of von-Mises-Fisher distributions on $S^m$ with $H$ components has density
\begin{eqnarray}
\label{vmf_mix}
f_{m+1}(x; \{ \pi_h, \mu_h, \kappa_h \}_{h=1}^{H}) = \sum_{h=1}^{H} \pi_h f_{m+1}(x;\mu_h,\kappa_h) .
\end{eqnarray}
The mixing proportions $\{\pi_h\}_{h=1}^{H}$ are non-negative and sum to 1 (i.e. $0 \le \pi_h \le 1, \sum_{h} \pi_h = 1$), and $\{\mu_h, \kappa_h\}_{h=1}^{H}$ are the parameters for the $H$ mixture components.
\\\\
Finite mixtures of von-Mises-Fisher distributions have found numerous applications, including clustering of high dimensional text data and gene expression \citep{banerjee2005} and clustering of online user behavior \citep{qin2016}. A natural question that arises is whether finite mixtures of von-Mises-Fisher distributions can approximate any continuous probability distribution on the hypersphere up to any desired level of accuracy. 
\\\\
In this paper, we provide an affirmative answer to this question. We prove that any continuous probability distribution on $S^m$ can be approximated by finite mixture of von-Mises-Fisher distributions in $\sup$ norm given enough mixture components, and each component is sufficiently concentrated at respective mean directions. Our proof utilizes the theory of approximation by spherical convolution \citep{menegatto1997}. 
\\\\
The paper is structured as follows. Section \ref{sec:background} provides relevant background that are needed for the proof of the main result. The main result is stated in Section \ref{sec:main} and is proved in Section \ref{sec:proof}.

\section{Background}
\label{sec:background}
This section provides the definitions of kernel function, spherical convolution and eigenfunction expansion which are needed for the proof of the main result. We refer the interested reader to \cite{menegatto1997} for detailed expositions of the theory.
\\\\
We denote the space of all continuous functions defined on the hypersphere $S^m$ by $C(S^m)$. Let $d\omega_m$ be the surface measure on $S^m$, and define $\omega_m := \int_{S^m} d \omega_m$. The uniform and the ${\cal L}^p$ norm on $S^m$ are defined as
$$ ||f||_{m, \infty} := \sup_{x \in S^m} |f(x)| $$
and
$$ ||f||_{m, p} := \bigg( \frac{1}{\omega_m} \int_{S^{m}} |f(x)|^{p} d \omega_m(x) \bigg)^{1/p} ,$$
respectively. In particular, the ${\cal L}^p$ space contains all functions defined on $S^m$ that are integrable with respect to $d \omega_m$. When no confusion arises, we let $V_m$ be any of the space above with corresponding norm $||\cdot||_{m}$ (i.e. $||\cdot||_{m} = ||\cdot||_{m,p}$ for $1 \le p < \infty$ or $||\cdot||_{m} = \sup_{x \in S^m} |f(x)| $). 
\\\\
We define the space ${\cal L}^{1,m}$ which consists of all measurable functions $K$ on $[-1,1]$ with norm
$$ ||K||_{1,m} := \frac{\omega_{m-1}}{\omega_m} \int_{-1}^{1} |K(t)| (1-t^2)^{(m-2)/2} dt < \infty .$$
Functions in the space ${\cal L}^{1,m}$ are called kernels. Let $\langle \cdot, \cdot \rangle$ be the inner product in $\mathbb{R}^{m+1}$, it is straight forward to show that for all $x \in S^m$, the following equality holds:
$$||K||_{1,m} := \frac{1}{\omega_m} \int_{S^{m}} |K(\langle x,y \rangle)| d \omega_m(y) .$$
The spherical convolution $K * f$ of a kernel $K$ in ${\cal L}^{1,m}$ with a function $f$ in $V_m$ is defined by
$$ (K * f)(x) := \frac{1}{\omega_m} \int_{S^m} K( \langle x,y \rangle )f(y) d\omega_m(y), \quad x \in S^{m}. $$
For a fixed kernel $K$, the mapping defined by the spherical convolution $f \rightarrow K * f$ for $f \in V_m$ has range in $V_m$.
\\\\
A useful property of spherical convolution is the Funk and Hecke's formula \citep{xu2000} for eigenfunction expansion of any kernel $K \in {\cal L}^{1,m}$. Let ${\cal H}_{k}^{m}$ be the space of all degree $k$ spherical harmonics in $m+1$ variables and let $N^{m}_{k}$ be its dimension \citep[Chapter~3]{reimer2012}. Let $Q_k^{(m-1)/2}$ be the Gegenbauer polynomial of degree $k$ normalized by $Q_k^{(m-1)/2}(1) = N_{k}^{m}$. The Gegenbauer polynomials are certain types of the Jacobi polynomials and are conveniently defined using generating functions \citep[Chapter~2]{reimer2012}. 
\\\\
The Funk and Hecke's formula states that for a kernel $K \in {\cal L}^{1,m}$ the following expansion holds:
$$ K * Y_k^{m} = a_k^{m}(K) Y_k^{m}, \quad K \in {\cal L}^{1,m}, Y_k^{m} \in {\cal H}_{k}^{m}, k=0,1,\ldots $$
In particular, the spherical harmonics $Y_k^{m}$ for $k=0,1,\ldots$ are the eigenfunctions associated with the kernel $K$, and the eigenvalues in the series expansion can be expressed in terms of Gegenbauer polynomials:
$$ a_k^m(K) = \frac{\omega_{m-1}}{\omega_m} \int_{-1}^{1} K(t) \frac{Q_k^{(m-1)/2}(t)}{Q_k^{(m-1)/2}(1)} (1-t^2)^{((m-2)/2} dt, \quad k=0,1,\ldots $$
In particular, we have
$$ a_0^{m}(K) = \frac{\omega_{m-1}}{\omega_m} \int_{-1}^{1} K(t) (1-t^2)^{((m-2)/2} dt . $$ 
\cite{menegatto1997} has investigated necessary and sufficient conditions under which a sequence of kernels $\{K_n\}_n$ in ${\cal L}^{1,m}$ has the property
$$ ||K_n * f - f||_{m} \rightarrow 0, \quad \forall f \in V_m $$
as $n \rightarrow \infty$. For non-negative kernels $\{K_n\}_n$, Theorem 3.4 of \cite{menegatto1997} provides sufficient conditions for the convergence of spherical convolutions $K_n * f \rightarrow f$, and is stated below.

\begin{lemma}
\label{conv1}
Let $\{K_n\}$ be a sequence of non-negative kernels in ${\cal L}^{1,m}$. Suppose 
\begin{enumerate}
\item $a_0^m(K_n) \rightarrow 1$ as $n \rightarrow \infty$;
\item $(\omega_{m-1} / \omega_m) \int_{-1}^{\rho} |K_n(t)| (1 - t^2)^{(m-2)/2} dt \rightarrow 0 $, for all $\rho \in (-1, 1),$
\end{enumerate}
then $||K_n * f - f||_{m} \rightarrow 0 $ as $n \rightarrow \infty$.
\end{lemma}

\section{Main Result}
\label{sec:main}
We state the main result concerning the approximating properties of the finite mixtures of von Mises-Fisher distributions in the form (\ref{vmf_mix}). Recall that the probability density function of the von Mises–Fisher distribution for the random $(m+1)$-dimensional unit vector $x$ is given by:
$$ f_{m+1}(x; \mu, \kappa) = c_{m+1}(\kappa) \exp(\kappa \langle x, \mu \rangle ) ,$$
where $\mu \in S^m$ is the mean direction and $\kappa > 0$ is the concentration parameter.
We define a sequence of kernels $ \{K_n\}_{n}$ in ${\cal L}^{1,m}$ by
\begin{eqnarray}
\label{kernel_seq}
K_n(t) = c_{m+1}(n) \exp(n t), \quad t \in [-1, 1] .
\end{eqnarray}
In particular, for any fixed $y \in S^m$,  
$$ K_n(\langle x, y \rangle) = c_{m+1}(n) \exp(\kappa \langle x, y \rangle ), \quad x \in S^m $$
is the density function of the von Mises-Fisher distribution with mean direction $y$ and concentration parameter $n$. For a fixed $y \in S^m$, $K_n(\langle \cdot, y \rangle)$ plays the role of a ``bump function'' and becomes increasingly concentrated on $y$ as $n$ increases. 
\\\\
We show that for any continuous probability density functions $f$ on $S^m$, we can construct a mixture of von Mises-Fisher distributions where each mixture component has the form $K_n(\langle \cdot, y_k \rangle)$ and $f$ can be approximated up to desired level of accuracy under the uniform norm.

\begin{theorem}
\label{dense2}
Let $f$ be a continuous probability density function on $S^m$, then given $\delta>0$, there exists integers $n$ and $N$, $y_1,y_2, \ldots y_N$ in $S^m$, $c_1, \ldots, c_N$ in $\mathbb{R}$ with $c_k > 0$ and $\sum_{k=1}^{N} c_k = 1$ such that
$$ \max_{x \in S^m}\bigg| f(x) - \sum_{k=1}^{N} c_k K_n(\langle x, y_k \rangle) \bigg| < \delta .$$
\end{theorem}

\section{Proof of Theorem 1}
\label{sec:proof}
In this section we first state and prove a few lemmas needed for the proof of Theorem \ref{dense2}. Recall that $V_m$ is the space of integrable functions on $S^m$ with respect to either the ${\cal L}^{p}$ norm or the uniform norm $||\cdot||_{m}$. We first show that for any function $f \in V_m$ the spherical convolution $K_n * f$ converges to $f$ in $||\cdot||_m$ norm. 
\begin{lemma}
\label{conv2}
$||K_n * f - f||_{m} \rightarrow 0 $ as $n \rightarrow \infty$ for all $f \in V_m$.
\end{lemma}
\begin{proof}
It is sufficient to verify conditions 1 and 2 in Lemma \ref{conv1}. For condition 1, since for non-negative kernel $K$ and for any fixed $x \in S^m$,
\begin{eqnarray*}
a_{0}^{m}(K) &=& \frac{\omega_{m-1}}{\omega_m} \int_{-1}^{1} K(t) (1-t^2)^{(m-2)/2} dt \\
   &=& \frac{1}{\omega_m} \int_{S^m} K(\langle x, y \rangle) d\omega_m(y) 
\end{eqnarray*} 
The last equality equals to 1 if $K$ is a probability density function.\\\\
For condition 2, we note that for any fixed $x \in S^m$,
\begin{eqnarray*}
\int_{-1}^{\rho} |K_n(t)| (1-t^2)^{(m-2)/2} dt &=& \frac{ \int_{-1}^{\rho} e^{nt} (1-t^2)^{(m-2)/2} dt}{\int_{-1}^{1} e^{nt} (1-t^2)^{(m-2)/2} dt} \\
 &=& \frac{ \int_{ \{y: \langle x, y \rangle \le \rho \}  } e^{n \langle x, y \rangle} d \omega_m(y) }{ \int_{ S^m  } e^{n \langle x, y \rangle} d \omega_m(y) }
\end{eqnarray*}
where the second equality is a result of applying a change of variable. Since $e^{n \langle x, y \rangle} \le e^{n \rho}$ if $\langle x, y \rangle < \rho$, the numerator above is bounded above by
\begin{eqnarray}
\label{ub_num}
\int_{ \{y: \langle x, y \rangle \le \rho \} } e^{n \langle x, y \rangle} d \omega_m(y) \le \omega_m(\{y: \langle x, y \rangle \le \rho\}) e^{n \rho} . 
\end{eqnarray}
To lower bound the denominator, we define the ball $B_{\delta}(x) := \{y \in S^m: \langle x, y \rangle \ge 1 - \delta\}$ where $1 - \delta > \rho$. Consequently,
\begin{align}
\label{lb_deno}
 \int_{ S^m  } e^{n \langle x, y \rangle} d \omega_m(y) \ge \int_{B_{\delta}(x)} e^{n \langle x, y \rangle} d \omega_m(y)   \\
       \ge e^{n(1-\delta)} \omega_m(B_{\delta}(x)) .
\end{align}
Therefore, combining the two inequalities (\ref{ub_num}) and (\ref{lb_deno}), we have
$$\int_{-1}^{\rho} |K_n(t)| (1-t^2)^{(m-2)/2} dt \le \frac{\omega_m(\{y: \langle x, y \rangle \le \rho\}) e^{n \rho}}{e^{n(1-\delta)} \omega_m(B_{\delta}(x))}.$$
Since $1 - \delta > \rho$, the RHS of the inequality above goes to 0 as $n \rightarrow \infty$.
\end{proof}

The following lemma concerning uniform approximation on $S^m$ by Riemann sums is useful.

\begin{lemma} 
\label{riemann_uniform}

Let $g(x, y):S^m\times S^m\to \mathbb R$ be a continuous function. Then for any $\delta>0$, there is a partition $\{U_1, \cdots, U_N\}$ of $S^m$ such that the integral $\int_{S^m} g(x, y)d\omega_m(y)$ can be uniformly approximated on $S^m$ by Riemann sums:
\begin{align*}
\max_{x\in S^m}\left|\int_{S^m}g(x, y)d\omega_m(y)-\sum_{k=1}^N g(x, y_k)\omega_m (U_k) \right|<\delta,
\end{align*}
for any $y_k \in U_k$, where each $U_k$ is connected.
\end{lemma}
\begin{proof}
For each $x'\in S^m$, there exists a neighborhood $\mathcal U_{x'}$ such that for $x\in \mathcal U_{x'}$, we have
\begin{align*}
\max_{y\in S^m}|g(x, y)-g(x', y)|<\frac{\delta}{3\omega_m}.
\end{align*}
Thus, for any $x \in \mathcal{U}_{x^{\prime}} $, we have
\begin{align*}
\left|\int_{S^m} g(x, y) d\omega_m(y)-\int_{S^m} g\left(x^{\prime}, y\right) d\omega_m(y)\right|
\le & \int_{S^m}\left|g(x, y)-g\left(x^{\prime}, y\right)\right| d\omega_m(y) \\
\le & \max _{y \in S^m}\left|g(x, y)-g\left(x^{\prime}, y\right)\right| \int_{S^m} d\omega_m(y)\\
<&\frac{\delta}{3}.
\end{align*}
There exists a partition $\{U_1, \cdots, U_{N'}\}$ of $S^m$ by standard spherical coordinates blocks such that $\int_{S^m}g(x', y)d\omega_m(y)$ can be approximated uniformly by Riemann sums:
\begin{align*}
\left|\int_{S^m} g(x', y)d\omega_m(y)-\sum_{k=1}^{N'}g(x', y_k)\omega_m(U_k)\right|<\frac{\delta}{3}
\end{align*}
for any $y_k\in U_k$.
Now, for $x \in \mathcal{U}_{x'}$, we have
\begin{align*}
\left|\int_{S^m} g(x, y) d\omega_m(y)-\sum_{k=1}^{N^{\prime}} g(x, y_{k})\omega_m(U_k)\right|
\le &\left|\int_{S^m} g(x, y) d\omega_m(y)-\int_{S^m} g(x^{\prime}, y) d\omega_m(y)\right| \\
&+\left|\int_{S^m} g(x^{\prime}, y) d\omega_m(y)-\sum_{k=1}^{N^{\prime}} g(x^{\prime}, y_{k})\omega_m(U_k)\right| \\
&+\left|\sum_{k=1}^{N^{\prime}} g(x^{\prime}, y_{k})\omega_m(U_k)-\sum_{k=1}^{N^{\prime}} g(x, y_{k})\omega_m(U_k)\right|\\
<&\delta.
\end{align*}

Since $\{\mathcal U_{x'}\}_{x'\in S^m}$ covers $S^m$, there exists a finite subcover $\{\mathcal U_{x_1}, \cdots, \mathcal U_{x_n}\}$. We can then find a common refinement of all the partitions used in the Riemann sums for $\int_{S_m}g(x_i, y)d\omega_m(y)$, $i=1, \cdots, n$. The claimed result follows immediately.
\end{proof}

The following result shows that any continuous function on $S^m$ can be uniformly approximated by linear combinations of $\{K_n(\langle \cdot, y_k\rangle )\}_{k}$ for $y_1, y_2, \ldots$ in $S^m$.

\begin{lemma} 
Let $f$ be a non-zero continuous function on $S^{m}$, then given $\delta>0, $ there exists integers $n$ and $N, y_{1}, y_{2}, \ldots y_{N}$ in $S^{m}, c_{1}, \ldots, c_{N}$ in $\mathbb{R}$ such that
\begin{align*}
\max _{x \in S^m}\left|f(x)-\sum_{k=1}^{N} c_{k} K_{n}\left(\left\langle x, y_{k}\right\rangle\right)\right|<\delta.
\end{align*}
\end{lemma}
\begin{proof}
By Lemma 2, there exists an integer $n$ such that
\begin{equation}\label{ineq1}
\max _{x \in S^m}\left|f(x)-\int_{S^{m}} K_{n}\left(\langle x, y\rangle \right) f(y) d \omega_{m}(y) \right|<\frac{\delta}{2}.
\end{equation}
On the other hand, by Lemma 3, there exists a partition $\{U_1, \cdots, U_N\}$ by connected sets of $S^m$ such that for any $x\in S^m$ and $y_k\in U_k$,
\begin{equation}
\label{riemann_approx}
\left|\int_{S^m}K(\langle x, y\rangle)f(y)d\omega_m(y)-\sum_{k=1}^N K_n(\langle x, y_k\rangle ) f(y_k)\omega_m(U_k)\right|<\frac{\delta}{2}.
\end{equation}
The result follows by combining \eqref{ineq1}, \eqref{riemann_approx}, and letting $c_{k}=f(y_k) \omega_m (U_{k}) $ for $k=1, \ldots, N$.

\end{proof}

\begin{proof} [Proof of Theorem 1]
It remains to carefully pick the points $y_{k} \in U_{k}$ to ensure that $\sum_{k=1}^{N} c_{k}=1$ in (\ref{riemann_approx}). This follows by applying the integral mean value theorem to each of the integrals
\begin{equation}
\int_{U_{k}} f(y) d \omega_m(y), \quad k=1, \ldots, N
\end{equation}
with connected $U_{k},$ and the fact that
\begin{equation*}
\sum_{k=1}^{N} \int_{U_k } f(y) d\omega_m(y)=1.
\end{equation*}

\end{proof}

\bibliographystyle{asa}
\bibliography{main}

\end{document}